\documentclass[11pt]{article}
\usepackage{color,xcolor}
\usepackage{amsfonts} 
\newcommand*\abs[1]{\lvert#1\rvert}

\usepackage{graphicx}
\usepackage{epstopdf}
\usepackage{amsmath}
\usepackage{amssymb}
\usepackage{multirow}
\usepackage{bbm}
\usepackage{color}
\usepackage{cite}
\usepackage{mathrsfs}
\usepackage{tabularx}
\allowdisplaybreaks[4]
\usepackage{booktabs}
\usepackage{mathtools}
\usepackage{amssymb,amsmath}
\usepackage{enumerate}
\usepackage{graphicx,amssymb,amsmath}
\usepackage{amsthm}
\usepackage{mathrsfs}
\usepackage{multirow}
\usepackage[numbers]{natbib}

\numberwithin{equation}{section}
\newtheorem{lem}{Lemma}[section]
\newtheorem{thm}{Theorem}[section]

\newtheorem{rem}{Remark}[section]

\begin{document}
\markboth{Huimin Huang, Wensheng Zhang}{Convergence Rates For Tikhonov Regularization of Coefficient Identification Problems in Robin-Boundary Equation}


\begin{center}{\bf Convergence Rates For Tikhonov Regularization of Coefficient Identification Problems in Robin-Boundary Equation}
\end{center}

\begin{center}
\centerline{\large Huimin Huang$^{1,2}$, Wensheng Zhang$^{1,2}$\footnote{Corresponding Author.~E-mail address:~zws@lsec.cc.ac.cn(W.S. Zhang)}}
\vspace{0.25cm}
{\small {\em
$^1$LSEC, ICMSEC, Academy of Mathematics and Systems Science, Chinese Academy of Sciences, Beijing 100190, China\\
$^2$School of Mathematical Sciences, University of Chinese Academy of Sciences, Beijing 100049, China}}
\end{center}
\begin{abstract}
This paper investigates the convergence rate for Tikhonov regularization of the problem of identifying the coefficient $a \in L^{\infty}(\Omega)$ in the Robin-boundary equation $-\mathrm{div}(a\nabla u)-bu=f,~ x \in \Omega \subset \mathbb R^M,~ M \geq 1$ and $u=0,~ x ~on~ \partial\Omega$, where $f(x)\in L^{\infty}(\Omega)$. Assume we only know the imprecise values of $u$ in the subset $\Omega_1 \subset \Omega$ given by $z^{\delta} \in {H}^1(\Omega_1)$, satisfies $\|u-z^{\delta}\|_{H^1(\Omega_1)}\leq \delta$. We assume $u$ satisfy the following boundary conditions on $\partial\Omega_1$:
\begin{align*}
\nabla u \cdot \vec{n}+\gamma u &=0~on~\partial\Omega_1,
\end{align*}
where $\vec{n}$ is the normal vector of $\partial\Omega_1$ and $\gamma>0$ is a constant. We regularize this problem by correspondingly minimizing the strictly convex functional:
  \begin{align*}
  \min \limits_{a \in \mathbb A} &\frac12 \int_{\Omega_1} a \abs {\nabla(U(a)-z^\delta)}^2
+\frac12\int_{\partial\Omega_1} a\gamma [U(a)-z^\delta]^2-\frac12 \int_{\Omega_1} b [U(a)-z^\delta]^2\\
&+ \rho \| a-a^* \|^2_{L^2(\Omega)},
  \end{align*}
  where $U(a)$ is a map for $a$ to the solution of the Robin-boundary problem, $\rho > 0$ is the regularization parameter and $a^*$ is a priori estimate of $a$. We prove that the functional attain a unique global minimizer on the admissible set. Further, we give very simple source condition without the smallness requirement on the source function which provide the convergence rate $O(\sqrt{\delta})$ for the regularized solution.
\end{abstract}


\section{Introduction}

Elliptic partial differential equation is used to describe the physical balance of steady state equation (see Ref.\cite{CFM01}), the partial differential equation of the problem is the definite condition of known equations to solve the problem of the solution, the inverse problem is to solve the problem by some known information of some unknown variables, since most of the inverse problem is ill-posed, solve the inverse problem of regularization method has become the main information (see Ref.\cite{CFM03},  ~\cite{CFM04}, ~\cite{CFM05}, ~\cite{CFM06}), Although there are many scholars devoted to the study of regularization methods (see Ref.\cite{CFM07},  ~\cite{CFM08}, ~\cite{CFM09}, ~\cite{CFM10},  ~\cite{CFM11}, ~\cite{CFM12}, ~\cite{CFM13},  ~\cite{CFM14}, ~\cite{CFM15}) and its references. However, the convergence rate of the regular solution is rarely studied (see Ref.\cite{CFM16},  ~\cite{CFM17}, ~\cite{CFM18}, ~\cite{CFM19},  ~\cite{CFM20}).

As far as we know, there are only papers by Engl and Kunisch (see Ref.\cite{CFM17}) devoted to the convergence speed of Tikhonov regularization for the above problem. In 1989, he used the output least squares method and Tikhonov regularization method to solve the nonlinear ill-posed problem, and obtained the convergence rate under certain source conditions. However, they face the difficulty of finding the global minimum when dealing with non-convex functionals. In addition, their source conditions are difficult to check and require high regularity in the coefficients sought.

In 2010, Dinh and Tran applied Tikhonov regularization to new convex energy functionals to solve the inverse problem and obtained the convergence rate of the method instead of using output least squares methods (see Ref.\cite{CFM19}). Their source conditions are easy to check and much weaker than those of Engl and Kunisch, because they remove the so-called condition that the source function is sufficiently small, which is generalized in nonlinear ill-posed regularization theory, they studied the elliptic equation
\begin{subequations}
\begin{align*}
-\mathrm{div}(a\nabla u) &=f~in ~ \Omega ,\\
u&=0 ~on~ \partial\Omega,
\end{align*}
\end{subequations}
proved the convergence rate of this equation is $\|a^{\delta}_{\rho}-a^+ \|_{L^2(\Omega)}=O(\sqrt{\delta})$.

In 2021, Wang and He (see Ref.\cite{CFM20}) studied the Dirichlet equation
 \begin{subequations}
\begin{align*}
-\mathrm{div}(a\nabla u)+cu &=f~in ~ \Omega ,\\
u&=0 ~on~ \partial\Omega,
\end{align*}
\end{subequations}
proved that the convergence rate of this equation is $\|a^{\delta}_{\rho}-a^+ \|_{L^2(\Omega)}=O(\sqrt{\delta})$.

Inspired by them, we studied the problem of identifying the coefficient $a$ in the Robin-boundary equation
\begin{subequations}\label{aa}
\begin{align}
-\mathrm{div}(a\nabla u)-bu&=f~in ~ \Omega ,\\
u&=0 ~on~ \partial\Omega,
\end{align}
\end{subequations}
where $a$ is the unknown coefficient satisfy $0<\underline{a}<a(x)<\overline{a}$, $b$ is a known coefficient satisfy $0<b(x)<\overline{b}$, and the observed value of $u$ named $z^\delta$ in $\Omega_1 \subset \Omega$ is used to approximate $u$. In this paper, we assume that $z^\delta$ and $u$ satisfy the following conditions on $\partial\Omega_1$:
\begin{align}\label{aam}
\nabla u \cdot \vec{n}+\gamma u =0 ~on~\partial\Omega_1,
\end{align}
where $\vec{n}$ is the normal vector of $\partial\Omega_1$ and $\gamma>0$ are given.

To overcome the shortcomings of the output least-squares method, we use the Tikhonov regularization which is applied to the new convex energy functional $G_{z^\delta}(a)$ and the convergence rate of its solution is calculated. The setting of convex energy functional is a difficult point to deal with this problem. It is necessary to ensure the convexity of the functional and easy to solve the convergence rate. Through multiple test, the functional:
\begin{align}\label{ab}
a\rightarrow G_{z^\delta}(a)= &\frac12 \int_{\Omega_1} a \abs {\nabla[U(a)-z^\delta]}^2
+\frac12\int_{\partial\Omega_1} a\gamma [U(a)-z^\delta]^2\nonumber\\
&-\frac12 \int_{\Omega_1} b [U(a)-z^\delta]^2,~~~~a \in \mathbb{A}
\end{align}
 is constructed for solving (\ref{aa}). Here, $U(a)$ is the coefficient-to-solution maps for (\ref{aa}) of $a$ to $U(a)$ with $\mathbb{A}$ being the admissible set, $z^\delta$ is the observed value of $u$ satisfies $\|u-z^{\delta}\|_{H^1(\Omega_1)}\leq \delta$, respectively. Then, we apply Tikhonov regularization to this functional and establish the convergence rate for this approach. Namely, we consider the minimization problem
\begin{align}\label{ad}
\min \limits_{a \in \mathbb A} &\frac12 \int_{\Omega_1} a \abs {\nabla[U(a)-z^\delta]}^2
+\frac12\int_{\partial\Omega_1} a\gamma [U(a)-z^\delta]^2-\frac12 \int_{\Omega_1} b [U(a)-z^\delta]^2\nonumber\\
&+ \rho \| a-a^* \|^2_{L^2(\Omega)},
\end{align}
 for identifying $a$ in (\ref{aa}), where $\rho > 0$ is the regularization parameter, and $a^*$ is a priori estimate of $a$.

Throughout this paper we assume that $\Omega_1$ is an open bounded connected domain in $\mathbb{R}^M,~~M>1$, with boundary $\partial\Omega_1$ and $f \in L^{\infty}(\Omega)$ are given. We use the standard notion of Sobolev
spaces $H^1(\Omega)$, $H^1_0(\Omega)$ and $W^{1,p}(\Omega)$ (see Ref.\cite{CFM01}). Moreover, in order to be no ambiguity, we write $\int _{\partial\Omega_1} \cdots$, $\int _{\Omega_1} \cdots$ and $\int _{\Omega} \cdots$instead of $\int _{\partial\Omega_1} \cdots \mathrm{d} s$, $\int _{\Omega_1} \cdots \mathrm{d} x$ and $\int _{\Omega} \cdots \mathrm{d} x$, respectively.

This paper is organized as follows. In section 2, we mainly prove the strictly convex functional $G_{z^\delta}(a)$ so that the minimization problem has a unique solution in the admissible set $\mathbb{A}$. We propose a relatively simple source condition in form, further prove the convergence and stability of the optimal solution. In section 3, the discussion on our convergence order is given. The discussion on our source condition is given in section 4.
\section{Convergence And Stability For Tikhonov Regularization Of The Diffusion Coefficient Identification Problems}
\subsection{Problem Setting}
We consider the problem of determining the coefficient $a=a(\cdot)\in L^\infty (\Omega)$ in the partial differential equation (\ref{aa}), assuming that $u$ is given in $\Omega_1$. To accurately describe the problem, we suppose that the function $u \in H^1(\Omega_1)$ is called a weak solution of (\ref{aa})-(\ref{aam}) if the function $u$ has following integral form:
\begin{align*}
\int_{\Omega_1} \left(-\mathrm{div}(a\nabla u)-bu\right)v
=\int_{\Omega_1}a\nabla u\nabla v-\int_{\partial\Omega_1}a\nabla u\vec{n}v-\int_{\Omega_1}buv
=\int_{\Omega_1} fv,
\end{align*}
for all $v \in H^1(\Omega_1)$.
That is
\begin{align}\label{cct}
\int_{\Omega_1}a\nabla u\nabla v+\int_{\partial\Omega_1}a\gamma uv-\int_{\Omega_1}buv
=\int_{\Omega_1} fv
\end{align}
if the coefficient $a$ belongs to the admissible set
\begin{align}\label{ct}
\mathbb A=\{ a(x) \in L^\infty (\Omega) \mid 0<\underline{a}<a(x)<\overline{a},~~\mathrm{a.e. ~on~} \Omega\}
\end{align}
where $\underline{a}$ and $\overline{a}$ are known constants. $b(x)$ belongs to the admissible set
\begin{align*}
\mathbb B=\{ b(x) \in L^\infty (\Omega)\mid 0 <b(x) <\bar b,~~\mathrm{a.e. ~on~} \Omega\}
\end{align*}
where $\bar b$ is a known constant, and $0< \overline{b} <\min\left\{\frac{\underline{a}}{C_P}, ~\frac{\underline{a}\widetilde{\gamma}}{C_F}\right\}$. Here $\widetilde{\gamma} =\min \{1,\gamma\}$ and $C_P$ and $C_F$ are known constants, depending only on $\Omega_1$, that appeared in the Poincar$\acute{e}$ inequality:
\begin{align}\label{at}
\int_{\Omega} v^2 \leq C_P\int_{\Omega} \abs {\nabla v}^2,~~ \forall ~v \in H_0^1(\Omega),
\end{align}
and $C_F$ appeared in the Friedrichs inequality:
\begin{align}\label{abb}
\parallel v \parallel^2_{H^1(\Omega_1)}\leq C_F\left[
\int_{\Omega} \abs{\nabla v }^2
+\left( \int_{\partial\Omega} v\right)^2\right],~~ \forall ~v \in H^1(\Omega),
\end{align}

So as to obtain the boundedness of $u$ in the $H^1(\Omega_1)$ norm,
 Then we have the following Lemma.
\begin{lem}\label{ta}
There is a unique weak solution in $H^1(\Omega_1)$ of (\ref{aa})-(\ref{aam}) which satisfies the inequality
\begin{align}\label{ba}
\| u \|_{H^1(\Omega_1)} \leq \frac 1 \alpha \| f \|_{L^2(\Omega)},
\end{align}
here
\begin{align}\label{aai}
\alpha=\frac{\underline{a}\widetilde{\gamma}-\overline{b}C_F}{C_F}>0.
\end{align}
where $C$ is a positive constant.
\end{lem}
\begin{proof}
Multiplying both sides of (\ref{aa}a) by $u$ and integrating by parts, we get
\begin{align*}
\int_{\Omega_1}a\abs{\nabla u}^2+\int_{\partial\Omega}a\gamma u^2-\int_{\Omega_1}bu^2
=\int_{\Omega_1} fu,
\end{align*}
for the left side of the equation we have
\begin{align}\label{aaz}
\int_{\Omega_1}a\abs{\nabla u}^2+\int_{\partial\Omega}a\gamma u^2-\int_{\Omega_1}bu^2
\geq \frac{\underline{a}\widetilde{\gamma}}{C_F}\| u \|^2 _{H^1(\Omega_1)}-\overline{b}\| u \|^2 _{H^1(\Omega_1)}.
\end{align}
For the right side of the equation, we use H\"{o}lder inequality,
\begin{align*}
&\int_{\Omega_1} fu
\leq\| f \|_{L^2(\Omega_1)}\| u \|_{L^2(\Omega_1)}
\leq\| f \|_{L^2(\Omega_1)}\| u \|_{H^1(\Omega_1)}.
\end{align*}

Because $f$ has a continuous boundary on $\partial\Omega_1$, that is $f$ is bounded in $\Omega_1$, then
\begin{align*}
\frac{\underline{a}\widetilde{\gamma}-\overline{b}C_F}{C_F}\| u \|_{H^1(\Omega_1)}
\leq\| f \|_{L^2(\Omega_1)},
\end{align*}

with (\ref{aai}) we get
\begin{align*}
\|  u \|_{H^1(\Omega_1)}
\leq \frac{1}{\alpha}\| f \|_{L^2(\Omega_1)}.
\end{align*}

\end{proof}
Therefore, define a nonlinear coefficient to solution operator $U:\mathbb A \subset L^{\infty}(\Omega)\rightarrow H^1(\Omega_1)$ to map the coefficient $a \in \mathbb A \subset L^{\infty}(\Omega)$ to the solution $U(a) \in H^1(\Omega_1)$ of the problem (\ref{aa})-(\ref{aam}).
The inverse problem is stated as given $\hat{u}:= U(a) \in H^1(\Omega_1)$, solves $a\in \mathbb A$.
\subsection{Tikhonov regularization}
We assume $\hat{u}$ is exact solution of (\ref{aa})-(\ref{aam}), and there exists some $a \in \mathbb A$ such that $\hat{u}=U(a)$,  where the set $\mathbb A$ is defined by (\ref{ct}) and $U(a)$ is the coefficient-to-solution mapping. We assume that instead of exact $\hat{u}$ we only know its observations  $z^\delta \in H^1(\Omega_1)$ which satisfies
\begin{align*}
\| \hat{u}-z^\delta \|_{H^1(\Omega_1)}\leq \delta,~~\delta>0.
\end{align*}

Our problem changes to reconstruct $a$ from $z^\delta$, For solving this problem we give that the minimum convex functional is
\begin{align}\label{ac}
G_{z^\delta}(a):=\frac12 \int_{\Omega_1} a \abs {\nabla[U(a)-z^\delta]}^2
+\frac12\int_{\partial\Omega_1} a\gamma [U(a)-z^\delta]^2
-\frac12 \int_{\Omega_1} b [U(a)-z^\delta]^2,
\end{align}
on set $\mathbb A$. However, since the problem is ill-posed, we shall use Tikhonov regularization to solve it in a stable way. Namely, we solve the minimization problem
\begin{align*}
\min \limits_{a \in \mathbb A} &\frac12 \int_{\Omega_1} a \abs {\nabla[U(a)-z^\delta]}^2+\frac12\int_{\partial\Omega_1} a\gamma [U(a)-z^\delta]^2-\frac12 \int_{\Omega_1} b [U(a)-z^\delta]^2\\
&+ \rho \| a-a^* \|^2_{L^2(\Omega)}.
\end{align*}
Note that $a^*$ is a priori estimate of $a$, it does not have to belong to the admissible set $\mathbb A$, it may be an element of $L^2(\Omega)$. In addition, $\| a-a^* \|^2_{L^2(\Omega)}$ is meaningful because $a \in \mathbb A$.

It is proved that the objective function of the problem (\ref{ad}) is weakly lower semi-continuous and strictly convex in the $L^2(\Omega)$-norm, the unique solution $a^\delta_\rho\in  L^2(\Omega)$ is obtained in a bounded closed sets of nonempty convexity, so $\mathbb A$ is a weakly compact set. In addition, we will prove that the problem (\ref{ad}) has a unique solution $a^\delta_\rho$ with respect to data $z^\delta$ in the $L^2(\Omega)$-norm and it is well posed.

Before our proof, the concept of $a^*$-minimum norm solution and some properties of $U(a)$ are introduced.
\begin{lem}\label{la}
If set
\begin{align*}
\bigvee \limits_{\mathbb A} (\hat{u}) = \{ a \in \mathbb A | U(a)=\hat{u} \}
\end{align*}
is nonempty, convex, bounded and closed in the $L^2(\Omega)$-norm, then exist a unique solution $a^+$ such that
\begin{align*}
\min \limits_{a \in \bigvee \limits_{\mathbb A}(\hat{u})} \| a-a^* \|^2_{L^2(\Omega)}=\| a^+-a^* \|^2_{L^2(\Omega)}
\end{align*}
holds, which is called the $a^*$-minimum norm solution of the identification problem.
\end{lem}
\begin{proof}
It is obvious that $\bigvee \limits_{\mathbb A} (\hat{u}) $ is a nonempty, convex and bounded set, so we need to proof it is closed.

Assume the sequence $\{a_n\} \subset \bigvee \limits_{\mathbb A} (\hat{u}) $ converges to $a$ in the $L^2(\Omega)$-norm, we proof $a \in \bigvee \limits_{\mathbb A} (\hat{u}) $.

In fact, for all $v \in H^1(\Omega_1)$ and $n \in \mathbb N$, we have
\begin{align*}
\int_{\Omega_1} fv &= \int_{\Omega_1}-\mathrm{div}(a_n \nabla U(a_n))v-\int_{\Omega_1} b U(a_n)v\\
&=\int_{\Omega_1} a_n\nabla U(a_n) \nabla v-\int_{\partial\Omega_1}a_n\nabla U(a_n)\vec{n}v-\int_{\Omega_1} b U(a_n)v\\
&=\int_{\Omega_1} a_n\nabla  \hat{u} \nabla v+\int_{\partial\Omega_1}a_n\gamma \hat{u}v-\int_{\Omega_1} b \hat{u} v,
\end{align*}
by the Fat\'{o}u Lemma and the definition of $\{a_n\}$, we get
\begin{align*}
&\lim \limits_{n \rightarrow \infty} \inf \int_{\Omega_1} a_n \nabla \hat{u} \nabla v=
\int_{\Omega_1} \lim \limits_{n \rightarrow \infty} \inf a_n \nabla \hat{u} \nabla v=
\int_{\Omega_1} a \nabla \hat{u} \nabla v,\\
&\lim \limits_{n \rightarrow \infty}\inf \int_{\partial\Omega_1}a_nk\hat{u}v=
\int_{\partial\Omega_1}\lim \limits_{n \rightarrow \infty}\inf a_n\gamma\hat{u}v=
\int_{\partial\Omega_1}a\gamma\hat{u}v,
\end{align*}
for all $v \in H^1(\Omega_1)$ as $n \rightarrow \infty$, that is
\begin{align*}
\int_{\Omega_1} fv=\int_{\Omega_1} a \nabla \hat{u} \nabla v+\int_{\partial\Omega_1}a\gamma\hat{u}v-\int_{\Omega_1} b \hat{u} v,
\end{align*}
so we obtain $\hat{u}=U(a)$ or $a \in \bigvee \limits_{\mathbb A} ( \hat{u})$.
\end{proof}

\begin{lem}\label{lb}
The mapping $U:\mathbb A \in L^{\infty}(\Omega)\rightarrow H^1(\Omega_1)$ is continuous Fr\'{e}chet differentible for every $a \in {\mathbb A}$, Fr\'{e}chet differential $U'(a)$ has following property: differential $\eta:=U'(a)h$ is the unique weakly solution of Robin-boundary problem in $H^1_0(\Omega_1)$ space,
\begin{subequations}\label{aae}
\begin{align}
-\mathrm{div} (a \nabla \eta) -b\eta &=\mathrm{div}(h\nabla U(a))~~in ~ ~\Omega_1,\\
\nabla \eta \cdot \vec{n}+\gamma\eta &=0~on~\partial\Omega_1,
\end{align}
\end{subequations}
where $h \in L^{\infty}(\Omega)$, $\vec{n}$ is the normal vector of $\partial\Omega_1$, that satisfies the variation equation
\begin{align}\label{ae}
\int_{\Omega_1}  a \nabla U'(a)h \nabla v +\int_{\partial\Omega_1} a\gamma U'(a)hv-\int_{\Omega_1} b U'(a)hv=-\int_{\Omega_1}h\nabla U(a)\nabla v-\int_{\partial\Omega_1} hk U(a)v,
\end{align}
for all $v \in H^1(\Omega_1)$, and the estimate
\begin{align}\label{ax}
\|U'(a)h\|_{H^1(\Omega_1)} \leq \frac1{\alpha\beta} \|f\|_{L^2(\Omega)}\|h\|_{L^\infty(\Omega)}
\end{align}
holds for all $h \in L^{\infty}(\Omega)$, where $\beta=\frac{\underline{a}\widetilde{\gamma}-\overline{b}C_P} {C_F(1+\gamma C_t^2)}>0$.
\end{lem}
where the $C_t$ come from the trace theorem, we introduce it by a remark:
\begin{rem}\label{rem}
Let $\Omega\subset \mathbb R^M$ be a bounded smooth domain, there exists a unique map $\chi: W^{1,p}(\Omega)\rightarrow L^p(\partial\Omega)$ and a positive constant $C_t=C(n,p,\Omega)$ such that:
\begin{subequations}\label{tr}
\begin{align}
&\chi u= u|_{\partial\Omega},~~ \forall ~ u \in W^{1,p}(\Omega)\cap C(\bar{\Omega}),\\
&\|\chi u \|_{L^p(\partial\Omega)}\leq C_t\| u \|_{W^{1,p}(\Omega)},~~ \forall ~ u \in W^{1,p}(\Omega).
\end{align}
\end{subequations}
\end{rem}

\begin{proof}
Because $U(a)$ satisty $\nabla u \cdot \vec{n}+\gamma u =0$ from integration by parts, we obtain the variation equation
\begin{align}\label{ag}
\int_{\Omega_1}  a \nabla \eta \nabla v +\int_{\partial\Omega_1} a\gamma \eta v-\int_{\Omega_1} b \eta v=-\int_{\Omega_1}h\nabla U(a)\nabla v-\int_{\partial\Omega_1} h\gamma U(a) v,
\end{align}
for all $v \in H^1(\Omega_1)$, from (\ref{aae}a) defines the unique solution $\eta:=\eta(h)\in H^1(\Omega_1)$.

We will see that $\eta=\eta(h)$ define the bounded linear operator from $L^\infty(\Omega)$ to $H^1(\Omega_1)$ when $a$ is fixed in set $\mathbb A$.

Because $\eta$ is the linear operator of $h$, taking $v=\eta$, with Remark 2.1 we can obtain the inequality,
\begin{align*}
\alpha\|\eta\|^2_{H^1(\Omega_1)}
&\leq \int_{\Omega_1} \underline{a} \abs{ \nabla \eta}^2+\int_{\partial\Omega_1} a\gamma \eta^2 -\overline{b}\int_{\Omega_1}\eta^2\\
&\leq-\int_{\Omega_1}h\nabla U(a)\nabla \eta-\int_{\partial\Omega_1} h\gamma U(a) \eta\\
&\leq \|h\|_{L^\infty(\Omega)}(1+\gamma C_t^2)\| U(a)\|_{H^1(\Omega_1)}\|\eta\|_{H^1(\Omega_1)},
\end{align*}
with (\ref{ba}) we have
\begin{align}\label{aac}
\|\eta\|_{H^1(\Omega_1)}\leq \frac1{\alpha\beta} \|f\|_{L^2(\Omega)} \|h\|_{L^\infty(\Omega)},
\end{align}
so $\eta=\eta(h)$ is a bounded linear operator in $L^\infty(\Omega) \rightarrow H^1(\Omega_1)$.
Now we proof $U(a)$ is Fr\'{e}chet differentiable.

In fact, we choose a $h \in L^\infty(\Omega)$ such that $a+h \in \mathbb A$, then
\begin{align*}
\int_{\Omega_1}fv&=\int_{\Omega_1} a \nabla U(a) \nabla v+\int_{\partial\Omega_1} a\gamma U(a) v -\int_{\Omega_1} b U(a)v\\
&=\int_{\Omega_1} (a+h) \nabla U(a+h) \nabla v +\int_{\partial\Omega_1} (a+h)\gamma U(a+h) v-\int_{\Omega_1} b U(a+h)v,
\end{align*}
for all $v \in H^1(\Omega_1)$, that is
\begin{align*}
&\int_{\Omega_1} (a+h)\nabla\left[U(a+h)-U(a)\right]\nabla v
+\int_{\partial\Omega_1} (a+h)\gamma \left[U(a+h)-U(a)\right] v\\
&-\int_{\Omega_1}b\left[ U(a+h)-U(a) \right]v\\
=&-\int_{\Omega_1}h\nabla U(a)\nabla v-\int_{\partial\Omega_1} h\gamma U(a) v\\
=&\int_{\Omega_1}a \nabla\eta\nabla v+\int_{\partial\Omega_1} a\gamma \eta v-\int_{\Omega_1} b\eta v,
\end{align*}
where $\eta$ is defined by the variational equation (\ref{ag}), then
\begin{align*}
&\int_{\Omega_1} (a+h)\nabla\left[U(a+h)-U(a)-\eta \right]\nabla v
+\int_{\partial\Omega_1} (a+h)\gamma \left[U(a+h)-U(a)-\eta\right] v\\
&-\int_{\Omega_1}b\left[ U(a+h)-U(a)-\eta \right]v
=-\int_{\Omega_1}h \nabla \eta \nabla v-\int_{\partial\Omega_1} h\gamma \eta  v,
\end{align*}
taking $v= U(a+h)-U(a)-\eta\in H^1(\Omega_1)$, then
\begin{align*}
\int_{\Omega_1}(a+h)\abs{\nabla v}^2 +\int_{\partial\Omega_1} (a+h)\gamma v^2-\int_{\Omega_1}b v^2
=-\int_{\Omega_1} h\nabla \eta \nabla v-\int_{\partial\Omega_1} h\gamma \eta v,
\end{align*}
so we know
\begin{align*}
\alpha\|v\|^2_{H^1(\Omega_1)}
&\leq (1+\gamma C_t^2)\|h\|_{L^\infty(\Omega)} \|\eta\|_{H^1(\Omega_1)}\|v\|_{H^1(\Omega_1)}\\
&\leq\frac{1}{\alpha\beta^2} \|f\|_{L^2(\Omega)}\|h\|^2_{L^\infty(\Omega)} \|v\|_{H^1(\Omega_1)},
\end{align*}
with (\ref{aac}) and Remark 2.1, that is
\begin{align*}
\|U(a+h)-U(a)-\eta\|_{H^1(\Omega_1)}\leq \frac1{\alpha\beta^2} \|f\|_{L^2(\Omega)}\|h\|^2_{L^\infty(\Omega)}
=O(\|h\|^2_{L^\infty(\Omega)}).
\end{align*}

From $\eta=\eta(h)$ is the bounded linear operator of $h$ in $L^\infty(\Omega) \rightarrow H^1(\Omega_1)$ we obtain $U:=\mathbb A \in L^{\infty}(\Omega)\rightarrow H^1(\Omega_1)$ is continuous Fr\'{e}chet differentiable for every $a \in {\mathbb A}$, and its Fr\'{e}chet differential $U'(a)h$ is $\eta$ for all $h \in L^{\infty}(\Omega)$.
\end{proof}
\begin{lem}\label{lc}
the convex functional
\begin{align}\label{ah}
G_{z^\delta}(a):=\frac12 \int_{\Omega_1} a \abs {\nabla[U(a)-z^\delta]}^2
+\frac12\int_{\partial\Omega_1} a\gamma [U(a)-z^\delta]^2-\frac12 \int_{\Omega_1} b [U(a)-z^\delta]^2
\end{align}
is convex in the convex set $\mathbb A$.
\end{lem}
\begin{proof}
because $\underline a \widetilde{\gamma}>\bar b C_F$, so
\begin{align*}
G_{z^\delta}(a)=&\frac12 \int_{\Omega_1} a \abs {\nabla[U(a)-z^\delta]}^2
+\frac12\int_{\partial\Omega_1} a\gamma [U(a)-z^\delta]^2-\frac12 \int_{\Omega_1} b [U(a)-z^\delta]^2\\
\geq&\frac{\underline{a}\widetilde{\gamma}-\overline{b}C_F}{2C_F} \|\nabla[U(a)-z^\delta]\|^2_{H^1(\Omega_1)}
\geq 0,
\end{align*}
for all $a \in \mathbb A$, then
\begin{align}\label{aaf}
G'_{z^\delta}(a)h
=& \int_{\Omega_1}a \nabla[U(a)-z^\delta]\nabla U'(a)h+
\frac12\int_{\Omega_1} \abs{\nabla [U(a)-z^\delta]}^2 h\nonumber\\
&+\int_{\partial\Omega_1}a\gamma [U(a)-z^\delta] U'(a)h+
\frac12\int_{\partial\Omega_1} \gamma[U(a)-z^\delta]^2 h\nonumber\\
&-\int_{\Omega_1} b [U(a)-z^\delta]U'(a)h,
\end{align}
for all $h \in L^{\infty}(\Omega)$, choosing $v=U(a)-z^\delta\in H^1(\Omega_1)$, the equation (\ref{aaf}) can be written as
\begin{align*}
G'_{z^\delta}(a)h
&=-\int_{\Omega_1} h\nabla U(a)\nabla v-\int_{\partial\Omega_1} h\gamma U(a) v+\frac12\int_{\Omega_1} \abs{\nabla v}^2h+\frac12\int_{\partial\Omega_1} \gamma v^2 h,
\end{align*}
with (\ref{ag}), so the second derivative of $G_{z^\delta}(a)$ is
\begin{align*}
G''_{z^\delta}(a)(h,k)
&=-\int_{\Omega_1}h\nabla U(a)\nabla U'(a)k-\int_{\partial\Omega_1} hk\gamma U(a) U'(a)\\
&=\int_{\Omega_1}hka \abs{\nabla U'(a)}^2+\int_{\partial\Omega_1} hk\gamma a (U'(a))^2 -\int_{\Omega_1}hkb(U'(a))^2
\end{align*}
for all $a \in \mathbb A$ and $h,k \in L^{\infty}(\Omega)$, let $h=k$, then
\begin{align*}
G''_{z^\delta}(a)h^2 \geq \alpha\|hU'(a)\|^2_{H^1(\Omega_1)},
\end{align*}
since $\alpha>0$, we get $G''_{z^\delta}(a)h^2 \geq 0$ for all $a \in \mathbb A$ and $h \in L^{\infty}(\Omega)$, so the functional $G_{z^\delta}(a)$ is convex in convex set $\mathbb A$.
\end{proof}

Now we proof the main conclusion of this subsection.
\begin{thm}\label{tb}
The problem (\ref{ad}) have a unique solution $a^{\delta}_{\rho}$.
\end{thm}
\begin{proof}
First, we prove that the functional $G_{z^\delta}(a)$ is continuous in set $\mathbb A$ with respect to $L^2(\Omega)$-norm.

Assume sequence $\{a_n\} \subset \mathbb A$ converge to $a$ in $L^2(\Omega)$-norm. For all $v \in H^1(\Omega_1)$ and $n \in \mathbb N$, we obtain
\begin{subequations}\label{ak}
\begin{align}
\int_{\Omega_1}fv&=\int_{\Omega_1}  a \nabla U(a) \nabla v -\int_{\partial\Omega_1}a\nabla U(a)\vec{n}v-\int_{\Omega_1} b U(a)v\nonumber\\
&=\int_{\Omega_1}  a \nabla U(a) \nabla v -\int_{\partial\Omega_1}a\gamma U(a)v-\int_{\Omega_1} b U(a)v,\\
\int_{\Omega_1}fv&=\int_{\Omega_1} a_n \nabla U(a_n) \nabla v- \int_{\partial\Omega_1}a_n\nabla U(a_n)\vec{n}v-\int_{\Omega_1} b U(a_n)v\nonumber\\
&=\int_{\Omega_1} a_n \nabla U(a_n) \nabla v- \int_{\partial\Omega_1}a_n\gamma U(a_n)v-\int_{\Omega_1} b U(a_n)v,
\end{align}
\end{subequations}
from (\ref{ba}) can obtain
\begin{align}\label{aj}
\| U(a_n) \|_{H^1(\Omega_1)} \leq \frac 1 \alpha\| f \|_{L^2(\Omega)},
\end{align}
so exist some $\theta \in H^1(\Omega_1)$ such that the subset of $\{U(a_n)\}$ defined by $\{U(a_{nk})\}$ weakly converge to $\theta$ in $H^1(\Omega_1)$, therefore
\begin{align}\label{aah}
&\int_{\Omega_1} a_{nk} \nabla U(a_{nk}) \nabla v -\int_{\partial\Omega_1}a_{nk}\gamma U(a_{nk})v-\int_{\Omega_1}b  U(a_{nk})v\nonumber\\
&-\left(\int_{\Omega_1} a \nabla \theta \nabla v -\int_{\partial\Omega_1}a\gamma \theta v-\int_{\Omega_1} b  \theta v\right)\nonumber \nonumber\\
=&\int_{\Omega_1}  (a_{nk}-a)\nabla U(a_{nk})\nabla v +\int_{\Omega_1} a (\nabla U(a_{nk})-\nabla \theta)\nabla v\nonumber\\
&-\int_{\partial\Omega_1}(a_{nk}-a)\gamma U(a_{nk})v
-\int_{\partial\Omega_1}a\gamma [U(a_{nk})-\theta]v
- \int_{\Omega_1} b  (U(a_{nk})-\theta) v.
\end{align}

Due to $\{a_n\}$ converges to $a$ in $L^2(\Omega)$-norm and (\ref{aj}), the first term and the third term on the right side of (\ref{aah}) tends to zero as $n \rightarrow \infty$.

On the other hand, since $\{U(a_n)\}$ converges weakly to $\theta$ in $H^1(\Omega)$ and $a,b$ is bounded, the second term, the forth term and the fifth term on the right side of (\ref{aah}) tends to zero as $n \rightarrow \infty$.

So we have
\begin{align}\label{al}
&\int_{\Omega_1} a_{nk} \nabla U(a_{nk}) \nabla v -\int_{\partial\Omega_1}a_{nk}\gamma U(a_{nk})v-\int_{\Omega_1}b  U(a_{nk})v\nonumber\\
&\rightarrow\int_{\Omega_1} a \nabla \theta \nabla v -\int_{\partial\Omega_1}a\gamma \theta v-\int_{\Omega_1} b  \theta v,
\end{align}
for all $v \in H^1(\Omega_1)$ as $n \rightarrow \infty$.

Combine (\ref{ak}) and (\ref{al}), we get $U(a)=\theta$.

Now we proof $G_{z^\delta}(a_n)\rightarrow G_{z^\delta}(a)$ when $n \rightarrow \infty$.

we have
\begin{align*}
G_{z^\delta}(a_n)
=&\frac12 \int_{\Omega_1} a_n\abs {\nabla[U(a_n)-z^\delta]}^2
+\frac12\int_{\partial\Omega_1} a_n\gamma [U(a_n)-z^\delta]^2
-\frac12 \int_{\Omega_1} b[U(a_n)-z^\delta]^2\\
=&\frac12\int_{\Omega_1}a_n\nabla [U(a_n)-2z^\delta]\nabla U(a_n)
+\frac12\int_{\Omega_1}a_n(\nabla z^\delta)^2\\
&+\frac12\int_{\partial\Omega_1}a_n\gamma[U(a_n)-2z^\delta]U(a_n)
+\frac12\int_{\partial\Omega_1}a_n\gamma(z^\delta)^2\\
&-\frac12 \int_{\Omega_1}b [U(a_n)-2z^\delta]U(a_n)
-\frac12 \int_{\Omega_1}b(z^\delta)^2,
\end{align*}
Due to (\ref{ak}a), we get
\begin{align*}
G_{z^\delta}(a_n)=&\frac12 \int_{\Omega_1}f[U(a_n)-2z^\delta]
+\frac12\int_{\Omega_1}a_n\abs{\nabla z^\delta}^2
+\frac12\int_{\partial\Omega_1}a_n\gamma(z^\delta)^2
-\frac12 \int_{\Omega_1}b(z^\delta)^2,
\end{align*}
use the same way, we know
\begin{align*}
G_{z^\delta}(a)
\rightarrow
\frac12 \int_{\Omega_1}f[U(a)-2z^\delta]+\frac12\int_{\Omega_1}a\abs{\nabla z^\delta}^2
+\frac12\int_{\partial\Omega_1}a\gamma(z^\delta)^2
-\frac12 \int_{\Omega_1}b(z^\delta)^2,
\end{align*}
so $G_{z^\delta}(a_n)\rightarrow G_{z^\delta}(a)$ as $n \rightarrow \infty$.

Now we proof (\ref{ad}) has a unique solution.

Because functional $G_{z^\delta}(a)$ is convex and continuous in set $\mathbb A$ respect to the $L^2(\Omega)$-norm and it is weakly lower semi-continuous. Therefore, the cost functional of (\ref{ad}) is convex and weakly lower semi-continuous in set $\mathbb A$. On the other hand, set $\mathbb A$ is nonempty convex closed bounded in the $L^2(\Omega)$-norm, it is weakly compact, so (\ref{ad}) has a unique solution.
\end{proof}

\begin{thm}\label{tc}
For a fixed regularization parameter $\rho >0$, let $\{z_n\}$ be a sequence which converges to $z^{\delta}$ in $H^1(\Omega_1)$, $\{a_n\}$ be the minimizers of the following problems:
\begin{align*}
\min \limits_{a \in \mathbb A} \frac12 \left\{\int_{\Omega_1} a \abs {\nabla[U(a)-z_n]}^2+\int_{\partial\Omega_1} a\gamma [U(a)-z_n]^2- \int_{\Omega_1} b[U(a)-z_n]^2\right\}+ \rho \| a-a^* \|^2_{L^2(\Omega)}
\end{align*}
so $\{a_n\}$ converges to the solution of (\ref{ad}) called $a^\delta_\rho$.
\end{thm}
\begin{proof}
From the definition of $\{a_n\}$, there is
\begin{align*}
G_{z_n}(a_n)+\rho \| a_n-a^* \|^2_{L^2(\Omega)} \leq G_{z_n}(a)+\rho \| a-a^* \|^2_{L^2(\Omega)}
\end{align*}
for each $a_n \in \mathbb A$, in the $L^2(\Omega)$ norm, we know
\begin{align*}
G_{z_n}(a)&= \frac12 \int_{\Omega_1} a \abs {\nabla[U(a)-z_n]}^2
+\frac12\int_{\partial\Omega_1} a\gamma [U(a)-z_n]^2
-\frac12 \int_{\Omega_1} b[U(a)-z_n]^2\\
&\leq (1+\gamma C_t^2)\int_{\Omega_1}a\abs{\nabla U(a)}^2+ \abs{\nabla z_n}^2,
\end{align*}
there exists a constant $C_1$ independent of $n$ satisty $\|z_n\|^2_{H^1(\Omega_1)} \leq C_1$ for all $n \in \mathbb N$, then
\begin{align*}
G_{z_n}(a_n)+\rho \| a_n-a^* \|^2_{L^2(\Omega)}\leq (1+\gamma C_t^2)\bar a[\| U(a)\|^2_{H^1(\Omega_1)}+C_1]+\rho \| a-a^* \|^2_{L^2(\Omega)},
\end{align*}
for all $a \in \mathbb A$.

It follows that $\{a_n\}$ is bounded and there exists a subsequence of $\{a_n\}$ denoted by $\{a_{nk}\}$ and a element $a^\delta_\rho \in L^2(\Omega)$ such that $\{a_{nk}\}$ weakly converges to $a^\delta_\rho$ in the $L^2(\Omega)$ norm.

According to $\mathbb A$ is convex and closed in $L^2(\Omega)$-norm, we get the conclusion that $a^\delta_\rho \in \mathbb A$. On the other side, because $G_{z^\delta}(\cdot)$ and norm $\| \cdot \|_{L^2(\Omega)}$ is weakly lower semi-continous, we have
\begin{subequations}\label{an}
\begin{align}
G_{z^\delta}(a^\delta_\rho) &\leq \liminf \limits_{n \in \mathbb N}G_{z^\delta}(a_n),\\
\| a^\delta_\rho-a^* \|^2_{L^2(\Omega)}&\leq \liminf \limits_{n \in \mathbb N} \| a_n-a^* \|^2_{L^2(\Omega)},
\end{align}
\end{subequations}
besides,
\begin{align*}
G_{z^\delta}(a_n)-G_{z_n}(a_n)
=&\int_{\Omega_1}a_n \nabla U(a_n)\nabla(z_n-z^\delta)
-\frac12 \int_{\Omega_1}a_n\nabla(z_n-z^\delta)\nabla(z_n+z^\delta)\\
&+\int_{\partial\Omega_1} a_n \gamma U(a_n)(z_n-z^\delta)
-\frac12\int_{\partial\Omega_1} a_n \gamma (z_n-z^\delta)(z_n+z^\delta)\\
&-\int_{\Omega_1}b U(a_n)(z_n-z^\delta)
+\frac12 \int_{\Omega_1}b(z_n-z^\delta)(z_n+z^\delta)\\
=&\int_{\Omega_1}a_n\nabla[U(a_n)-z_n] \nabla(z_n-z^\delta)
+\frac12 \int_{\Omega_1}a_n\abs{\nabla(z_n -z^\delta)}^2 \\
&+\int_{\partial\Omega_1} a_n \gamma [U(a_n)-z_n](z_n -z^\delta)
-\frac12\int_{\partial\Omega_1} a_n\gamma (z_n-z^\delta)^2\\
&-\int_{\Omega_1}b[U(a_n)-z_n](z_n -z^\delta)
-\frac12 \int_{\Omega_1}b(z_n-z^\delta)^2,
\end{align*}
because $\{z_n\}$ converges to $z^\delta$ in $H^1(\Omega_1)$ as $n \rightarrow \infty$, we get
\begin{align}\label{ao}
\liminf \limits_{n \in \mathbb N}G_{z^\delta}(a_n) = \liminf \limits_{n \in \mathbb N}G_{z_n}(a_n).
\end{align}
and
\begin{align}\label{ap}
&G_{z^\delta}(a^\delta_\rho)+\rho\|a^\delta_\rho-a^* \|^2_{L^2(\Omega)}\nonumber\\
\leq& \liminf \limits_{n \in \mathbb N}G_{z^\delta}(a_n)
+\rho \liminf \limits_{n \in \mathbb N} \| a_n-a^* \|^2_{L^2(\Omega)}\nonumber\\
=&\liminf \limits_{n \in \mathbb N}G_{z_n}(a_n)
+\rho \liminf \limits_{n \in \mathbb N} \| a_n-a^* \|^2_{L^2(\Omega)}\nonumber\\
\leq& \liminf \limits_{n \in \mathbb N}\left[ G_{z_n}(a_n)+\rho\| a_n-a^* \|^2_{L^2(\Omega)} \right] \nonumber\\
\leq& \limsup \limits_{n \in \mathbb N}\left[ G_{z_n}(a_n)+\rho\| a_n-a^* \|^2_{L^2(\Omega)} \right] \nonumber\\
\leq& \limsup \limits_{n \in \mathbb N}\left[ G_{z_n}(a)+\rho\| a-a^* \|^2_{L^2(\Omega)} \right] \nonumber\\
=&G_{z^\delta}(a)+\rho\|a-a^* \|^2_{L^2(\Omega)}
\end{align}
for all $a \in \mathbb A$ from (\ref{an}) and (\ref{ao}). So $a^\delta_\rho$ is the unique solution of (\ref{ad}).

We proof $\{a_{n}\}$ converges to $a^\delta_\rho$ in $L^2(\Omega)$-norm by contradiction.

In fact, assume $\{a_n\} \nrightarrow a^\delta_\rho$, this and (\ref{an}b) follow that
\begin{align}\label{aag}
\varepsilon:&=\limsup \limits_{n \in \mathbb N}\|a_n-a^* \|^2_{L^2(\Omega)}> \|a^\delta_\rho-a^* \|^2_{L^2(\Omega)}
\end{align}
there exists a subsequence $\{a_{nj}\}$ of $\{a_n\} $ such that $\{a_{nj}\}$ weakly converges to $a^\delta_\rho$ in $L^2(\Omega)$ -norm and
\begin{align}\label{aq}
\|a^\delta_\rho-a^* \|^2_{L^2(\Omega)} \rightarrow \varepsilon.
\end{align}
taking $a=a^\delta_\rho$ in (\ref{ap}), we get
\begin{align}\label{ar}
G_{z^\delta}(a^\delta_\rho)+\rho\|a^\delta_\rho-a^* \|^2_{L^2(\Omega)}\nonumber
=&\lim \limits_{nj \in \mathbb N}\left[ G_{z_{nj}}(a_{nj})+\rho\|a_{nj}-a^* \|^2_{L^2(\Omega)} \right]\nonumber\\
\geq & \liminf \limits_{nj \in \mathbb N} G_{z_{nj}}(a_{nj})+\rho \varepsilon,
\end{align}
by (\ref{aq}). Combining (\ref{ao}), (\ref{aag})and (\ref{ar}), we obtain
\begin{align*}
\liminf \limits_{nj \in \mathbb N} G_{z^\delta}(a_{nj})
\leq G_{z^\delta}(a^\delta_\rho)+\rho \left( \|a^\delta_\rho -a^* \|^2_{L^2(\Omega)}-\varepsilon \right)
< G_{z^\delta}(a^\delta_\rho),
\end{align*}
that is in contradiction with (\ref{an}a), so $\{a_n\} \rightarrow a^\delta_\rho$ as $n\rightarrow \infty$.
\end{proof}

\section{convergence rate}
Because $L^\infty(\Omega)=L^1(\Omega)^*  \subset L^\infty(\Omega)^*$, any $a\in L^\infty(\Omega)$ can be cinsidered as an element in $L^\infty(\Omega)^*$, the dual space of $L^\infty(\Omega)$, by
\begin{align}\label{au}
\langle a,h\rangle_{(L^\infty(\Omega)^*,L^\infty(\Omega))}=\int_{\Omega}ah,\quad
\|a \|_{L^\infty(\Omega)^*} \leq \mathrm{mes}(\Omega)\|a \|_{L^\infty(\Omega)},
\end{align}
for all $h \in L^\infty(\Omega)$. Besides, the mapping
\begin{align*}
U'(a):L^\infty(\Omega)\rightarrow H^1(\Omega_1),
\end{align*}
is a continuous linear operator for $a \in \mathbb A$. We denote
\begin{align*}
U'(a)^*: H^1(\Omega_1)^*=H^{-1}(\Omega_1)\rightarrow   L^\infty(\Omega)^*,
\end{align*}
is the dual operator of $U'(a)$. Therefore,
\begin{align}\label{av}
\langle U'(a)^* \omega ^*, h\rangle_{(L^\infty(\Omega)^*,L^\infty(\Omega))}=
\langle\omega ^*,U'(a)h\rangle_{(H^{-1}(\Omega),H^1(\Omega_1))}
\end{align}
for all $\omega ^* \in H^{-1}(\Omega_1)$ and $h \in L^\infty(\Omega)$.

Denote $C^\infty_c(\Omega)$ is the set of infinityly differentiable functions with compact support lying in $\Omega$, then the Lemma 2.5 is given.

\begin{lem}\label{ld}
Denote $(\psi_\rho)_{\rho \in (0,1)}$ is a family of functions of $C^\infty_c(\Omega_1)$ space which are bounded in $H^{1}(\Omega)$-norm, $(\phi_\rho)_{\rho \in (0,1)}$ is a family of unique solutions of the problems
\begin{subequations}\label{ai}
\begin{align}
-\mathrm{div}(a\nabla\phi_\rho)-b\phi_\rho&=\psi_\rho-\Delta \psi_\rho~in ~ \Omega_1 ,\\
\phi_\rho&=0 ~on~ \partial\Omega_1,
\end{align}
\end{subequations}
in $H_0^1(\Omega_1)$ with $a\in \mathbb A$ is a fixed element.

So there exist a constant $C_2>0$ such that
\begin{align*}
\max\left\{ \int_{\Omega_1}\phi_\rho^2,  \int_{\Omega_1}\abs{\nabla\phi_\rho}^2\right\} \leq C_2^2,\quad \forall \rho \in (0,1)
\end{align*}
\end{lem}
\begin{proof}
The problem (\ref{ai}) has a unique weak solution $\phi_\rho \in H_0^1(\Omega_1)$ in the sense that
\begin{align*}
\int_{\Omega_1}a \nabla\phi_\rho\nabla v -\int_{\Omega_1}b\phi_\rho v
=\int_{\Omega_1}(\psi_\rho-\Delta \psi_\rho)v
=\int_{\Omega_1}\psi_\rho v+\int_{\Omega_1} \nabla \psi_\rho \nabla v,
\end{align*}
by inegration by parts for all $v \in H_0^1(\Omega_1)$, choosing $v=\phi_\rho$, we have
\begin{align*}
\underline{a}\int_{\Omega_1}\abs{\nabla \phi_\rho}^2
-\bar b\int_{\Omega_1}\phi_\rho^2
\leq \int_{\Omega_1}\psi_\rho \phi_\rho+\int_{\Omega_1} \nabla \psi_\rho \nabla \phi_\rho
\end{align*}
from the definition of $\mathbb A$, use (\ref{at}b) and Cauchy-Schwarz inequality, we get
\begin{align*}
(\underline{a}-\bar b C_P)\int_{\Omega_1}\abs{\nabla \phi_\rho}^2
\leq& \int_{\Omega_1}\psi_\rho \phi_\rho+\int_{\Omega_1} \nabla \psi_\rho \nabla \phi_\rho\\
\leq&\left(\int_{\Omega_1}\psi_\rho^2\right)^{\frac12} \left(\int_{\Omega_1}\phi_\rho^2\right)^{\frac12}+
\left(\int_{\Omega_1} \abs{\nabla \psi_\rho}^2\right)^{\frac12}
\left(\int_{\Omega_1} \abs{\nabla \phi_\rho}^2\right)^{\frac12}\\
\leq& \frac1{4\varepsilon_3}\int_{\Omega_1}\psi_\rho^2+\varepsilon_3\int_{\Omega_1}\phi_\rho^2+
\frac1{4\varepsilon_4} \int_{\Omega_1} \abs{\nabla \psi_\rho}^2+
\varepsilon_4\int_{\Omega_1} \abs{\nabla \phi_\rho}^2
\end{align*}
with $\varepsilon_3=\frac{\underline{a}-\bar b C_P}{4 C_P}, ~\varepsilon_4=\frac{\underline{a}-\bar b C_P}{4}$, then
\begin{align*}
\int_{\Omega_1} \abs{\nabla \phi_\rho}^2 \leq \frac{2 C_P}{(\underline{a}-\bar b C_P)^2}\int_{\Omega_1}\psi_\rho^2+
\frac{2}{(\underline{a}-\bar b C_P)^2}\int_{\Omega_1} \abs{\nabla \psi_\rho}^2.
\end{align*}

By the definition of $\psi_\rho$, there exist a $C_3$ such that
\begin{align*}
\|\psi_\rho \|^2_{H^1(\Omega_1)}  \leq C_3
\end{align*}
holds for all $\rho \in (0,1)$. Therefore,
\begin{align*}
\int_{\Omega_1} \abs{\nabla \phi_\rho}^2 &\leq
\frac{2 (C_P+1)}{(\underline{a}-\bar b C_P)^2}C_3,\quad
\int_{\Omega_1}  \phi_\rho^2 \leq \frac{2C_P (C_P+1)}{(\underline{a}-\bar b C_P)^2}C_3,
\end{align*}
so there exist a $C_2>0$ such that
\begin{align*}
\max\left\{ \int_{\Omega_1}\phi^2_\rho,  \int_{\Omega_1}\abs{\nabla \phi_\rho}^2\right\} \leq C_2^2,\quad \forall \rho \in (0,1)
\end{align*}
\end{proof}

\begin{thm}\label{td}
If there exist a function $\omega^* \in H^{-1}(\Omega_1)$ such that the source condition
\begin{align}\label{aw}
a^+ - a^*=U'(a^+)^*\omega^*
\end{align}
holds, then
\begin{subequations}\label{as}
\begin{align}
\|a^{\delta}_{\rho}-a^+ \|_{L^2(\Omega)}=O(\sqrt{\delta})~and~
\|U(a^{\delta}_{\rho})-z^\delta \|_{H^1(\Omega_1)}=O(\delta),
\end{align}
\end{subequations}
as $\rho \rightarrow 0$ and $\rho \sim \delta$.
\end{thm}
\begin{proof}
From the definition of regularization solution $a^{\delta}_{\rho}$ and $\|\hat{u}-z^\delta \|_{H^1(\Omega_1)}\leq \delta$, we have
\begin{align*}
&G_{z^\delta}(a^{\delta}_{\rho})+\rho \|a^{\delta}_{\rho}-a^* \|^2_{L^2(\Omega)}\\
\leq &G_{z^\delta}(a^+)+\rho \|a^+-a^* \|^2_{L^2(\Omega)}\\
=&\frac12 \int_{\Omega_1} a^+ \abs {\nabla[U(a^+)-z^\delta]}^2
+\frac12\int_{\partial\Omega_1} a^+\gamma [U(a^+)-z^\delta]^2
-\frac12 \int_{\Omega_1} b[U(a^+)-z^\delta]^2\\&+\rho \|a^+-a^* \|^2_{L^2(\Omega)}\\
\leq & \frac{(1+\gamma)\bar a}2{\delta}^2+\rho \|a^+-a^* \|^2_{L^2(\Omega)},
\end{align*}
so
\begin{align*}
&G_{z^\delta}(a^{\delta}_{\rho})+\rho \|a^+-a^{\delta}_{\rho} \|^2_{L^2(\Omega)}\\
\leq& \frac{(1+\gamma)\bar a}2{\delta}^2+\rho \|a^+-a^* \|^2_{L^2(\Omega)}-\rho \|a^{\delta}_{\rho}-a^*\|^2_{L^2(\Omega)}+
\rho \|a^+-a^{\delta}_{\rho} \|^2_{L^2(\Omega)}\\
\leq&\frac{(1+\gamma)\bar a}2 {\delta}^2+2\rho \langle a^+-a^*,a^+ -a^{\delta}_{\rho}\rangle_{L^2(\Omega)},
\end{align*}
from (\ref{au}), (\ref{av}) and (\ref{aw}) we have
\begin{align*}
2\rho \langle a^+-a^*,a^+ -a^{\delta}_{\rho}\rangle_{L^2(\Omega)}
=&2\rho \int_{\Omega} (a^+-a^*)(a^+ -a^{\delta}_{\rho})\\
=&2\rho \langle a^+-a^*,a^+ -a^{\delta}_{\rho} \rangle_{(L^{\infty}(\Omega)^*,L^{\infty}(\Omega))}\\
=&2\rho \langle U'(a^+)^*\omega^*,a^+ -a^{\delta}_{\rho} \rangle_{(L^{\infty}(\Omega)^*,L^{\infty}(\Omega))}\\
=&2\rho \langle \omega^*,U'(a^+)(a^+ -a^{\delta}_{\rho})\rangle _{(H^{-1}(\Omega_1),H^1(\Omega_1))},
\end{align*}
there exist a element $\omega \in H_0^1(\Omega_1)$ such that
\begin{align*}
\langle \omega^*,U'(a^+)(a^+ -a^{\delta}_{\rho})\rangle _{(H^{-1}(\Omega_1),H^1(\Omega_1))}
=\langle \omega,U'(a^+)(a^+ -a^{\delta}_{\rho})\rangle _{H^1(\Omega_1)},
\end{align*}
by the Riesz representation theorem, for fixed $\rho \in (0,1)$, we choose $\psi_\rho \in C^\infty_c(\Omega_1)$ such that
\begin{align*}
\|\omega-\psi_\rho\|_{H^1(\Omega_1)}\leq \rho,
\end{align*}
so
\begin{align*}
&2\rho \langle a^+-a^*,a^+ -a^{\delta}_{\rho}\rangle_{L^2(\Omega)}\\
=&2\rho \langle\psi_\rho,U'(a^+)(a^+ -a^{\delta}_{\rho})\rangle _{H^{1}(\Omega_1)}+ 2\rho \langle\omega-\psi_\rho,U'(a^+)(a^+ -a^{\delta}_{\rho})\rangle _{H^{1}(\Omega_1)}\\
\leq&2\rho \langle\psi_\rho,U'(a^+)(a^+ -a^{\delta}_{\rho})\rangle _{H^{1}(\Omega_1)}+ 2\rho^2\|U'(a^+)(a^+ -a^{\delta}_{\rho})\| _{H^{1}(\Omega_1)}\\
=&\Sigma +\Lambda,
\end{align*}
$a^+$ and $a^{\delta}_{\rho}$ belong to $\mathbb A$ and (\ref{ax}) yield
\begin{align*}
\Lambda
\leq 2\rho^2\cdot\frac1{\alpha\beta} \|f\|_{L^2(\Omega)}\|a^+ -a^{\delta}_{\rho}\|_{L^\infty(\Omega)}
\leq \frac{4 \bar a}{\alpha\beta}\rho^2\|f\|_{L^2(\Omega)}.
\end{align*}

We consider the following Dirichlet problem:
\begin{subequations}\label{ay}
\begin{align}
-\mathrm{div}( a^+ \nabla \phi_\rho)-b \phi_\rho&=\psi_\rho-\Delta \psi_\rho~in ~ \Omega_1 ,\\
\phi_\rho&=0 ~on~ \partial\Omega_1,
\end{align}
\end{subequations}
because $a^+ \in \mathbb A$ and $\psi_\rho \in C^\infty_c(\Omega_1)$, we conclude that (\ref{ay}) has a unique weak solution $\phi_\rho \in H^1_0(\Omega_1)$ such that
\begin{align*}
\int_{\Omega_1}a^+\nabla\phi_\rho\nabla v -\int_{\Omega_1} b\phi_\rho v =\int_{\Omega_1}(\psi_\rho-\Delta \psi_\rho)v,
\end{align*}
for all $v \in H^1(\Omega_1)$.
Assume $v=U'(a^+)(a^+ -a^{\delta}_{\rho})$, then
\begin{align*}
\Sigma & =2\rho \langle\psi_\rho,v\rangle _{H^{1}(\Omega_1)}=2\rho\left(\int_{\Omega_1}\psi_\rho v +\int_{\Omega_1}\nabla\psi_\rho\nabla v \right)\\
&=2\rho\int_{\Omega_1}(\psi_\rho-\Delta\psi_\rho)v=2\rho\left(\int_{\Omega_1}a^+\nabla\phi_\rho\nabla v -\int_{\Omega_1} b\phi_\rho v\right),
\end{align*}
that is
\begin{align*}
\Sigma =2\rho\left[\int_{\Omega_1}a^+\nabla\phi_\rho\nabla U'(a^+)(a^+ -a^{\delta}_{\rho})-\int_{\Omega_1} b\phi_\rho U'(a^+)(a^+ -a^{\delta}_{\rho})\right]
\end{align*}
from (\ref{ae}) we have
\begin{align*}
\Sigma =&-2\rho\int_{\Omega_1}(a^+ -a^{\delta}_{\rho})\nabla U(a^+)\nabla \phi_\rho\\
=&-2\rho\int_{\Omega_1}a^+\nabla U(a^+)\nabla\phi_\rho+
2\rho\int_{\Omega_1}b U(a^+) \phi_\rho\\
&+2\rho\int_{\Omega_1}a^{\delta}_{\rho}\nabla U(a^+)\nabla \phi_\rho
-2\rho\int_{\Omega_1}b U(a^+) \phi_\rho\\
=&2\rho\int_{\Omega_1}a^{\delta}_{\rho}\nabla \phi_\rho\nabla \left[U(a^+)-U(a^{\delta}_{\rho}) \right]
-2\rho\int_{\Omega_1}b  \phi_\rho \left[U(a^+)-U(a^{\delta}_{\rho}) \right]\\
=&2\rho\int_{\Omega_1}a^{\delta}_{\rho}\nabla \phi_\rho\nabla \left[ U(a^+)-z^{\delta} \right]+2\rho\int_{\Omega_1}a^{\delta}_{\rho}\nabla \phi_\rho\nabla \left[z^{\delta}-U(a^{\delta}_{\rho}) \right]\\
&-2\rho\int_{\Omega_1}b \phi_\rho
\left[U(a^+)-z^{\delta}\right]
-2\rho \int_{\Omega_1}b \phi_\rho \left[z^{\delta}-U(a^{\delta}_{\rho})\right],
\end{align*}
where the third equals sign comes from
\begin{align*}
\int_{\Omega_1}fv
=\int_{\Omega_1} a^+ \nabla U(a^+) \nabla v -\int_{\Omega_1} b U(a^+)v
=\int_{\Omega_1} a^{\delta}_{\rho} \nabla U(a^{\delta}_{\rho}) \nabla v -\int_{\Omega_1}bU(a^{\delta}_{\rho})v,
\end{align*}
using the Cauchy-Schwarz inequation and $U(a^+)=\hat{u}$, we have
\begin{align*}
\Sigma
\leq&2\bar a\rho \left[\int_{\Omega_1} \abs{\nabla U(a^+)-\nabla z^{\delta} }^2\right]^{\frac12}
\left[\int_{\Omega_1}\abs{\nabla\phi_\rho}^2\right]^{\frac12}\\
& +2 \bar a \rho\left[\int_{\Omega_1} \abs{\nabla z^{\delta}- \nabla U(a^{\delta}_{\rho})}^2\right]^{\frac12} \left[\int_{\Omega_1}\abs{\nabla \phi_\rho}^2\right]^{\frac12}\\
&+2\bar b \rho\left\{\int_{\Omega_1}\left[U(a^+)-z^{\delta}\right]^2 \right\}^{\frac12}\left[\int_{\Omega_1}\phi_\rho^2\right]^{\frac12}
+2\bar b\rho \left\{\int_{\Omega_1}\left[z^{\delta}-U(a^{\delta}_{\rho}] \right)^2 \right\}^{\frac12}\left[\int_{\Omega_1} \phi_\rho^2\right]^{\frac12}\\
\leq&2\bar a \rho\delta\left[\int_{\Omega_1}\abs{\nabla \phi_\rho}^2\right]^{\frac12}+
\varepsilon_5\int_{\Omega_1}\abs{\nabla U(a^{\delta}_{\rho}) -\nabla z^{\delta}}^2+\frac{{\bar a}^2}{\varepsilon_5}\rho^2\int_{\Omega_1} \abs{\nabla\phi_\rho}^2\\
&+2\bar b\rho \delta\left[\int_{\Omega_1}\phi_\rho^2\right]^{\frac12}
+\varepsilon_5\int_{\Omega_1}\left[z^{\delta}-U(a^{\delta}_{\rho}) \right]^2+\frac{{\bar b}^2}{\varepsilon_5}\rho^2\int_{\Omega_1}\phi_\rho^2,
\end{align*}
$\|\omega-\psi_\rho\|_{H^1(\Omega_1)} \leq \rho$ implies $(\psi_\rho)_{\rho \in (0,1)}$ is bounded in the $H^1(\Omega_1)$-norm. Using Lemma 3.1, we have
\begin{align*}
\Sigma \leq&2 C_2\bar a  \rho  \delta+2 C_2\bar b \rho  \delta+\frac{ C_2^2 {\bar a}^2}{\varepsilon_5} \rho^2 +\frac{ C_2^2{\bar b}^2}{\varepsilon_5} \rho^2 +\varepsilon_5 \|U(a^{\delta}_{\rho})-z^\delta\|^2_{H^1(\Omega_1)},
\end{align*}
from (\ref{abb}) we obtain
\begin{align*}
G_{z^\delta}(a^{\delta}_{\rho}) 
\geq\frac{\underline a-\bar bC_F}{2} \|\nabla U(a^{\delta}_{\rho})-\nabla z^\delta\|^2_{L^2(\Omega_1)}
\geq\frac{\lambda}{2} \| U(a^{\delta}_{\rho})- z^\delta\|^2_{H^1(\Omega_1)},
\end{align*}
where $\lambda=\frac{\underline a-\bar bC_F}{C_F}$, so we have
\begin{align*}
&\frac{\lambda}{2} \|U(a^{\delta}_{\rho})- z^\delta\|^2_{H^1(\Omega_1)}+\rho\|a^+-a^{\delta}_{\rho} \|^2_{L^2(\Omega)}\\
\leq&G_{z^\delta}(a^{\delta}_{\rho})+\rho\|a^+-a^{\delta}_{\rho} \|^2_{L^2(\Omega)}\\
\leq& \frac{(1+\gamma)\bar a}2\delta^2+\frac{4 \bar a}{\alpha\beta}\rho^2\|f\|_{L^2(\Omega)}
+2 C_2(\bar a+\bar b)  \rho \delta
+\frac{C_2^2({\bar a}^2+ {\bar b}^2)}{\varepsilon_5} \rho^2+\varepsilon_5 \|U(a^{\delta}_{\rho})-z^\delta\|^2_{H^1(\Omega_1)},
\end{align*}
with $\varepsilon_5=\frac{\lambda}{4}$ we obtain
\begin{align*}
&\frac{\lambda}{4} \|U(a^{\delta}_{\rho})-z^\delta\| ^2_{H^1(\Omega_1)}+ \rho\|a^+-a^{\delta}_{\rho} \|^2_{L^2(\Omega)}\\
\leq&\frac{(1+\gamma)\bar a}2\delta^2+\frac{4 \bar a \|f\|_{L^2(\Omega)}}{\alpha\beta}\rho^2
+2 C_2(\bar a+\bar b)  \rho \delta
+\frac{4 C_2^2({\bar a}^2+ {\bar b}^2)}{\alpha} \rho^2\\
=&O(\delta^2),
\end{align*}
as $\delta\rightarrow 0$ and $\rho \sim \delta$, that is $$\|U(a^{\delta}_{\rho})-z^\delta\| ^2_{H^1(\Omega_1)}=O(\delta),~~~ \|a^+-a^{\delta}_{\rho} \|^2_{L^2(\Omega)}=O(\sqrt{\delta}).$$

Because $\|U(a^+)-z^\delta\|_{H^1(\Omega_1)} \leq \delta$, we have
\begin{align*}
 \|U(a^{\delta}_{\rho})-U(a^+)\|_{H^1(\Omega_1)}
 \leq
 \|U(a^{\delta}_{\rho})-z^\delta\|_{H^1(\Omega_1)}
 +\|z^\delta-U(a^+)\|_{H^1(\Omega_1)}= O(\delta).
\end{align*}
\end{proof}
\section{the discussion of condition}
Now we discuss the source condition (\ref{aw}). This is a quite weak source condition and it especially does not require any smoothness of $a^+$. 
We note that the source condition (\ref{aw}) is fulfilled if there exists a function $\omega \in H^1_0(\Omega_1)$ such that
\begin{align*}
\langle a^+-a^*,h\rangle_{L^2(\Omega)}=
\langle\omega,U'(a^+)h\rangle_{H^1(\Omega_1)},
\end{align*}
for all $h \in L^{\infty}(\Omega)$.

We see that this weaker condition is satisfied under a quite weak hypothesis about the regularity of the sought coefficient.
However, as we can see in the proof of theorem 3.1 and to
obtaining the convergence rate $O(\sqrt{\delta})$, we only need the following weaker condition: there exists a function $\omega \in H^1_0(\Omega_1)$ such that
\begin{align}\label{az}
\langle a^+-a^*,a^+-a^{\delta}_{\rho}\rangle_{L^2(\Omega)}= \langle\omega,U'(a^+)(a^+-a^{\delta}_{\rho})\rangle_{H^1(\Omega_1)},
\end{align}
for all $a^{\delta}_{\rho}$ is the regularization solution of (\ref{ad}).

To further analyze the source condition we assume that the sought coefficient is known on the boundary $\partial \Omega$ of the domain $\Omega$. Therefore, the admissible set of coefficients is restricted to
\begin{align*}
\mathbb A=\{ a \in L^\infty (\Omega)| 0<\underline{a}<a(x)<\overline{a}~~on ~~\Omega ~~and ~~a(x)=m(x)~~on ~~\partial \Omega_1\},
\end{align*}
where $m \in C(\partial\Omega_1)$ is given. The regular assumption of the coefficients in $H^1 (\Omega)$ is introduced to guarantee that there exists their trace on the boundary $\partial\Omega$.

In the following, as $a^*$ is only a priori estimate of $a^+$, for simplicity, we assume that $a^* \in H^1 (\Omega)$.
\begin{thm}\label{te}
Assume $a^+ \in W^{1,\infty}(\Omega)$ and $\partial \Omega$ is piecewise smooth, suppose $\hat{u}\in C^2( \overline{\Omega_1})$ and $\abs{\nabla \hat{u}}\neq 0$, where $\hat{u}=U(a^+)$. then condition (\ref{az}) is fulfilled and the Tikhonov regularization convergence rate $O(\surd{\delta})$ is obtained.
\end{thm}
To prove theorem 3.1 we need the following auxiliary results.
\begin{lem}\label{lle}
Let $a^+$ belong to $W^{1,\infty}(\Omega)$ and the boundary is piecewise smooth, assume that there exists a function $v \in H^1(\Omega_1)$ such that
\begin{align}\label{aaa}
-\nabla U(a^+)\nabla v=a^+-a^*,
\end{align}
on $\Omega$. Then, condition (\ref{az}) is fulfilled.
\end{lem}
\begin{proof}
It follows from (\ref{aaa}) that
\begin{align}\label{aab}
\langle a^+-a^*,a^+-a^{\delta}_{\rho}\rangle_{L^2(\Omega)}= & \langle-\nabla U(a^+)\nabla v, a^+-a^{\delta}_{\rho}\rangle_{L^2(\Omega_1)}\nonumber\\
=&-\int_{\Omega_1}(a^+-a^{\delta}_{\rho}) \nabla U(a^+)\nabla v,
\end{align}

So we notice the boundary is piecewise smooth and the solution of $U(a^+)$ and unique solution of the following Dirichlet problem $\eta=U'(a^+)(a^+-a^{\delta}_{\rho}) \in H^2(\Omega_1) \cap H_0^1(\Omega_1)$
\begin{subequations}\label{bb}
\begin{align}
-\mathrm{div}(a^+ \nabla \eta)-b\eta&=\mathrm{div}((a^+-a^{\delta}_{\rho}) \nabla U(a^+))~in ~ \Omega_1 ,\\
\nabla \eta \cdot \vec{n}+\gamma\eta &=0 ~on~ \partial\Omega_1.
\end{align}
\end{subequations}

Both sides of (\ref{bb}a) multiply by $v \in H^1(\Omega_1)$ and part integral, we obtain
\begin{align}\label{be}
&\int_{\Omega_1}a^+\nabla\eta\nabla v+\int_{\partial\Omega_1}a^+\nabla \eta \vec{n}v-\int_{\Omega_1} b \eta v\nonumber\\
=&-\int_{\Omega_1}(a^+-a^{\delta}_{\rho}) \nabla U(a^+)\nabla v-\int_{\partial\Omega_1}(a^+-a^{\delta}_{\rho}) \nabla U(a^+) \vec{n}v\nonumber\\
=&\langle a^+-a^*,a^+-a^{\delta}_{\rho}\rangle_{L^2(\Omega)},
\end{align}
by $\nabla \eta \cdot \vec{n}+\gamma\eta =0$ on $\partial \Omega_1$, we have
\begin{align*}
\langle a^+-a^*,a^+-a^{\delta}_{\rho}\rangle_{L^2(\Omega)}
=&\int_{\Omega_1}a^+\nabla [U'(a^+)(a^+-a^{\delta}_{\rho})]\nabla v-\int_{\Omega_1} b U'(a^+)(a^+-a^{\delta}_{\rho})v\\
&+\int_{\partial\Omega_1}a^+ \gamma U'(a^+)(a^+-a^{\delta}_{\rho}) v,
\end{align*}
the linear function $\Gamma: H^1(\Omega_1)\rightarrow \mathbb R$ defined by
\begin{align*}
\Gamma(\xi)=\int_{\Omega_1}a^+ \nabla\xi\nabla v+\int_{\partial\Omega_1}a^+\gamma \xi v-\int_{\Omega_1}b \xi v,
\end{align*}
is bounded in $H^1(\Omega_1)$ respect $H^1(\Omega_1)$-norm. because
\begin{align*}
|\Gamma(\xi)|
\leq &\|a^+\|_{L^{\infty}(\Omega_1)}
\|\xi\|_{H^1(\Omega_1)}\|v\|_{H^1(\Omega_1)}
+\gamma\|a^+\|_{L^{\infty}(\Omega_1)}\|\xi\|_{L^2(\partial\Omega_1)}\|v\|_{L^2(\partial\Omega_1)}\\
&+\|b\|_{L^{\infty}(\Omega_1)}\|\xi\|_{L^2(\Omega_1)}\|v\|_{L^2(\Omega_1)}\\
\leq& \bar a\|\xi\|_{H^1(\Omega_1)}\|v\|_{H^1(\Omega_1)}
+C_t^2\bar a\gamma\|\xi\|_{H^1(\Omega_1)}\|v\|_{H^1(\Omega_1)}\\
&+\bar b\|\xi\|_{H^1(\Omega_1)}\|v\|_{H^1(\Omega_1)}\\
=&(\bar a+C_t^2\bar a\gamma+\bar b)\|\xi\|_{H^1(\Omega_1)}\|v\|_{H^1(\Omega_1)}.
\end{align*}

Then we use Riesz representation theorem, there exist a ${\omega} \in H^1_0(\Omega_1)$ such that
\begin{align*}
\Gamma(\xi)=\langle \omega, \xi\rangle_{H^1(\Omega_1)},
\end{align*}
for all $\xi \in H^1(\Omega_1)$, that is
\begin{align*}
\langle a^+-a^*,a^+-a^{\delta}_{\rho}\rangle_{L^2(\Omega)}=
\Gamma[U'(a^+)(a^+-a^{\delta}_{\rho})]=\langle \omega, U'(a^+)(a^+-a^{\delta}_{\rho})\rangle_{H^1(\Omega_1)}.
\end{align*}
\end{proof}

Notice that the requirement $a^+ \in W^{1,\infty}(\Omega)$ ensure the solution $ \hat{u}=U(a^+) \in H^1(\Omega_1)$, the existence of function $v \in H^1(\Omega_1)$ in Theorem 4.1 is guaranteed by the following result (see Ref.\cite{CFM12}).

\begin{lem}\label{llm}
Assume $\hat{u} \in C^2(\overline{\Omega})$ and $\abs{\nabla \hat{u}}\neq 0$ on $\overline{{\Omega}}$, then for any element $\tilde{a} \in H^1(\Omega)$, there exist $v \in H^1(\Omega)$, satisfing the equation $\nabla \hat{u} \nabla v =\widetilde{a}$ and the estimation $\|v\|_{H^1(\Omega_1)} \leq C\|\tilde{a}\|_{H^1(\Omega)}$, where the constant $C$ is independent of  $\tilde{a}$.
\end{lem}

Then Theorem 4.1 directly follows from Lemma 4.1 and Lemma 4.2.


\begin{thebibliography}{99}
\bibitem{CFM01}Z. Wu, J. Yin and C. Wang, Elliptic \& Parabolic Equations, \emph{World Scientific} (2006).

\bibitem{CFM03}G. Alessandrini, An identification problem for an elliptic equation in two variables, \emph{Ann. Mat. Pura. Appl.} \textbf{145(1)} (1986) 265-296.

\bibitem{CFM04}J. Zou, Numerical methods for elliptic inverse problem, \emph{Int. J. Comput. Math.} \textbf{70(2)} (1998) 211-232.

\bibitem{CFM05}K. Ito and K. Kunisch, On the injectivity and linearization of the coefficient-to-solution mapping for elliptic boundary value problem, \emph{J. Math. Anal. Appl.} \textbf{188(3)} (1994) 1040-1066.

\bibitem{CFM06}I. Knowles, Parameter identification for elliptic problems, \emph{J. Comput. Appl. Math.} \textbf{131(1-2)} (2001) 175-194.

\bibitem{CFM07}H. Banks and K. Kunisch, Estimation Techniques for Distributed Parameter Systems, \emph{Birkh\"{a}user Boston} (1989).

\bibitem{CFM08}G. Chavent, Nonlinear Least Squares for Inverse Problems: Theoretical Foundations and Step-by-step Guide for Applications, \emph{Scientific Computation} (2010).

\bibitem{CFM09}G. Chavent and K. Kunisch, The output least squares identifiability of the diffusion coefficient from an $H^1$-observation in a 2d elliptic equation, \emph{ESIAM: Contr. Optim. Ca.} \textbf{8} (2002) 423-440.

\bibitem{CFM10}F. Colonius and K. Kunisch, Output least squares stability in elliptic systems, \emph{Appl. Math. Opt.} \textbf{19} (1989) 33-63.

\bibitem{CFM11}K. Ito and K. Kunisch, Lagrange multiplier approach to variational problems and application, \emph{SIAM. Philadelphia} \textbf{15} (2008).

\bibitem{CFM12}R.Kohn and B. Lowe, A variational method for parameter identification, \emph{RAIRO Mod\'{e}l. Math. Anal. Numer.} \textbf{22} (1988) 119-158.

\bibitem{CFM13}G. Richter, An inverse problem for the steady state diffusion equation, \emph{SIAM. J. Appl. Math.} \textbf{41} (1981) 210-221.

\bibitem{CFM14}N. Sun, Inverse Problems in Groundwater Modeling, \emph{Kluwer Academic} (1994).

\bibitem{CFM15}W. Yeh, Review of parameter identification procedures in ground water hydrology: the inverse problem, \emph{Water. Resour. Res.} \textbf{22(2)} (1986) 95-108.

\bibitem{CFM16}N. Dinh and N. Tran, Convergence rates for total variation regularization of coefficient identification problems in elliptic equations, \emph{Inverse. Probl.} \textbf{27(7)} (2011) 593-616.

\bibitem{CFM17}H. Engl, K. Kunisch and A. Neubauer, Convergence rates for Tikhonov regularization of non-linear ill-posed problems, \emph{Inverse. Probl.} \textbf{5} (1989) 523-540.

\bibitem{CFM18}M. Hanke, A regularizing levenberg-marquardt scheme with applications to inverse groundwater filtration problem, \emph{Inverse. Probl.} \textbf{13(1)} (1997) 79-95.

\bibitem{CFM19}N. Dinh and N. Tran, Convergence rates for Tikhonov regularization of coefficient identification problems in Laplace-type equations, \emph{Inverse. Probl.} \textbf{26(12)} (2010) 1-22.

\bibitem{CFM20}Q. Wang and Q. He, Convergence rate of regularization solution for coefficient identification problem in elliptic equations, \emph{Journal of Luoyang Institute of Science and Technology (Natural Science Edition)} \textbf{31(2)} (2021) 74-82.

\end{thebibliography}
\end{document}